\documentclass[12pt]{amsart}

\usepackage{amsfonts,amssymb,amsmath,amsthm,mathtools,mathabx,bm}
\usepackage{url}
\usepackage{bbm}
\theoremstyle{plain}
\newtheorem{theorem}{Theorem}[section]
\newtheorem{lemma}[theorem]{Lemma}
\newtheorem{corollary}[theorem]{\bf Corollary}
\newtheorem{remark}[theorem]{\bf Remark}

\newtheorem{proposition}[theorem]{\bf Proposition}

\newcommand \Sym{ \mbox{Sym}}

\newcommand \Appr{ \rm{Appr}}

\newcommand \Iso{{\rm{ Iso}}}
\newcommand \Th{{\rm{ Th}}}

\newcommand \FL{\mathcal{L}_{\omega \omega}}
\newcommand \IL{\mathcal{L}_{\omega_1 \omega}}
\newcommand \Mod{\rm{Mod}}
\newcommand \tp{\rm{tp}}

\newcommand \rk{{\rm{ rk}}}

\newcommand \bPi{\bm{\Pi}}
\newcommand \bSigma{\bm{\Sigma}}

\newcommand \KK{\mathcal{K}}

\newcommand \PP{\mathcal{P}}
\newcommand \VV{\mathcal{V}}

\newcommand \NN{\mathbb{N}}
\newcommand \Seq{\mathbb{N}^{<\mathbb{N}}}
\newcommand \RR{\mathbb{R}}
\newcommand \QQ{\mathbb{Q}}

\newcommand \UU{\mathbb{U}}
\newcommand \UUU{\mathcal{U}}

\newcommand{\cl}[2][]{\overline{#2}^{#1}}

\mathtoolsset{showonlyrefs} 

\title{Isomorphism of locally compact Polish metric structures}

\author[M. Malicki]{Maciej Malicki}
\address{Institute of Mathematics, Polish Academy of Sciences, ul. Sniadeckich 8, Warsaw, Poland}
\email{mamalicki@gmail.com}

\begin{document}
\maketitle

\begin{abstract}
We study the isomorphism relation on Borel classes of locally compact Polish metric structures. We prove that isomorphism on such classes is always classifiable by countable structures (equivalently: Borel reducible to graph isomorphism), which implies, in particular, that isometry of locally compact Polish metric spaces is Borel reducible to graph isomorphism. We show that potentially $\bPi^0_{\alpha+1}$ isomorphism relations are Borel reducible to equality on hereditarily countable sets of rank $\alpha$, $\alpha \geq 2$. We also study approximations of the Hjorth-isomorphism game, and formulate a condition ruling out classifiability by countable structures.
\end{abstract}

\section{Introduction}

An equivalence relation $E$ on a Polish space $X$ is Borel reducible to an equivalence relation $F$ on a Polish space $Y$ if there is a Borel mapping $f:X \rightarrow Y$ such that $$x_1 E x_2 \mbox{ iff } f(x_1) F f(x_2).$$ The notion of Borel reducibility can be thought of as a general framework for measuring complexity of various notions of isomorphism. For instance, Ornstein's celebrated theorem says that isomorphism of Ber\-noulli shifts can be characterized by their entropy, i.e., by real numbers. This translates into the language of Borel reducibility as the statement that the isomorphism equivalence relation on the space of Bernoulli shifts is Borel reducible to the equality relation on $\RR$ (via an appropriate coding of Bernoulli shifts as elements of a Polish space, and a Borel mapping assigning to shifts their entropy). In other words, this relation is \emph{smooth}: it admits invariants that are elements of a Polish space. However, there are many interesting classification results that do not yield as concrete invariants. Halmos and von Neumann proved that isomorphism of measure preserving transformations with discrete spectrum is reducible to equality on countable subsets of the unit circle (via a mapping assigning to such transformations their spectrum). And Kechris showed that orbit equivalence relations induced by actions of locally compact Polish groups are Borel reducible to relations with countable classes, i.e., they are \emph{essentially countable}.

As a matter of fact, all these results have a common feature: they say that the involved equivalence relations are \emph{classifiable by countable structures}, i.e., they are Borel reducible to the isomorphism relation on a Borel class of countable structures in the sense of model theory. It should not be surprising that in this setting tools coming from logic play a vital role. It has been known for a long time that there are deep connections between model theory of the infinitary logic $\IL$ and descriptive set theory -- e.g., Scott analysis or the Lopez-Escobar theorem. And Hjorth's theory of turbulence, inspired by Scott analysis, is a prominent example of this phenomenon in the theory of Borel reducibility. Another important result, due to Hjorth and Kechris, characterizes in model-theoretical terms essential countability of isomorphism on Borel classes of countable structures: it is essentially countable iff there is a countable fragment $F$ of $\IL$ such that for every structure $M$ in the class there is a tuple $\bar{a}$ such that $\Th_F(M,\bar{a})$ is $\aleph_0$-categorical.

Very recently, successful attempts have been made at extending this approach to continuous logic. In this setting, Polish (i.e., separable and complete) metric structures play the role of countable structures. A continuous $\IL$ logic was first studiedb by Ben Yaacov and Iovino in \cite{BeIo}, and a continuous version of Scott analysis was developed in \cite{BeDoNiTs}, laying the foundations for descriptive set theoretic applications. And in \cite{HaMaTs}, Hjorth and Kechris' characterization of essential countability has been generalized to locally compact Polish metric structures, leading to a model-theoretic proof of Kechris' theorem on orbit equivalence relations induced by actions of locally compact Polish groups.

In the present paper, we continue this line of research. We show in Theorem \ref{th:isoTalpha} that isomorphism classes of locally compact Polish metric structures can be characterized by hereditarily countable sets built out of closed subsets of appropriate type spaces, and this implies that isomorphism on Borel classes of locally compact Polish metric structures is always classifiable by countable structures (Theorem \ref{th:CtbleModels}). In particular, we confirm a conjecture of Gao and Kechris from \cite{GaKe}, where they asked whether isometry of locally compact Polish metric spaces is Borel reducible to graph isomorphism (Theorem \ref{th:Isometry}). Next, we perform a fine-grained analysis of Borel isomorphism relations along the lines of \cite{HjKe}. %And in Corollary \ref{co:Sigma0alpha} we show that Polish topologies on isomorphism invariant sets of Polish metric structures are defined by fragments of $\IL$. This, together with the fact that $G_\delta$ isomorphism classes are isomorphism classes of atomic models, yields in 
Generalizing in Theorem \ref{th:Potentially} results from \cite{HjKeLo}, we prove that potentially $\bPi^0_{\alpha+2}$ isomorphism on a Borel class of locally compact Polish metric structures is Borel reducible to equality on hereditarily countable sets of rank $\alpha+1$, $\alpha \geq 1$. 

In the last part of the paper, we turn to equivalence relations that are \emph{not} classifiable by countable structures. Lupini and Panagiotopolous \cite{LuPa} recently developed a game-theoretic approach that (with an aid of Hjorth's theory of turbulence) gives an interesting sufficient condition for not being classifiable in this way. We introduce and study a hierarchy of games that are finer and finer approximations of the Hjorth-isomorphism game considered in \cite{LuPa}. We show that in the case of isomorphism of countable structures, these games capture information contained in families $Th^\alpha(M)$ (where $Th^0(M)$ is the theory of $M$, $Th^1(M)$ is the family of theories of structures $(M,\bar{a})$, etc). We also provide in Theorem \ref{th:NotClassifiable} a sufficient condition ruling out classifiability by countable structures, and, in Theorem \ref{th:NotClassifiablePialpha}, by countable structures with isomorphism of a given Borel complexity.

\section{Notation and basic facts}

In this section, we briefly discuss basics of infinitary continuous logic $\IL$. For a more detailed treatment, the reader is referred to \cite{BeDoNiTs} and \cite{HaMaTs}. A \emph{modulus of continuity} is a continuous function $\Delta \colon [0, \infty) \to [0, \infty)$ satisfying for all $r, s \in [0, \infty)$:
\begin{itemize}
	\item $\Delta(0)=0$,
	\item $\Delta(r) \leq \Delta(r+s) \leq \Delta(r)+\Delta(s)$.
\end{itemize}
Suppose that $\Delta$ is a modulus of continuity and that $(X, d_X)$ and $(Y,d_Y)$ are metric spaces. We say that a map $f: X \to Y$ \emph{respects} $\Delta$ if
\[ d_Y(f(x_1),f(x_2)) \leq \Delta(d_X(x_1, x_2)) \quad \text{ for all } x_1, x_2 \in X.\]

A \emph{signature} $L$ is a collection of predicate and function symbols and as is customary, we treat constants as $0$-ary functions. Throughout the paper, we assume that $L$ is countable. To each symbol $P$ are associated its \emph{arity} $n_P$ and its \emph{modulus of continuity} $\Delta_P$, and, if $P$ is a predicate, its \emph{bound}, i.e., a compact interval $I_P \subseteq \RR$. In a metric structure $M$ with \emph{complete} metric $d$, extended to finite or infinite tuples by putting, for $m,n \leq \omega$, $\bar{a} \in M^m$, $\bar{b} \in M^n$,
\begin{equation*}
	d(\bar a, \bar b) = \max \{d(a_i, b_i): i<\min(m,n) \},
\end{equation*}
predicate symbols are interpreted as real-valued functions of the appropriate arity respecting the modulus of continuity and the bound; similarly for function symbols. 

Terms and atomic formulas in infinitary continuous logic $\IL(L)$ in signature $L$ are defined in the usual way. Other formulas are built using:
\begin{itemize}
	\item Finitary connectives: if $\phi$ and $\psi$ are formulas and $r \in \QQ$, then $\phi + \psi$, $r \phi$, and $\phi \vee \psi$ are formulas. Here $\phi \vee \psi$ is interpreted as $\max(\phi, \psi)$, we also define $\phi \wedge \psi \coloneqq -(-\phi \vee -\psi) = \min(\phi, \psi)$, $\phi \dotdiv \psi=\max(\phi-\psi,0)$. The constant $1$ is a formula. 
	
	\item Quantifiers: if $\phi(x, \bar y)$ is a formula, then $\sup_x \phi$ and $\inf_x \phi$ are formulas.
	
	\item Infinitary connectives: if $\{\phi_n(\bar x) : n \in \NN\}$ are formulas with the same finite set of free variables $\bar x$ that \emph{respect a common continuity modulus and bound}, then $\bigvee_n \phi_n$ and $\bigwedge_n \phi_n$ are formulas. The symbol $\bigvee$ is interpreted as a countable supremum and $\bigwedge$ is interpreted as a countable infimum. The condition that we impose ensures that the interpretations of these formulas are still bounded, uniformly continuous functions.
\end{itemize}

The interpretations of formulas in a metric structure $M$ are defined in the usual way. It is important to keep in mind that to any formula are associated its modulus of continuity and bound that can be calculated from its constituents. If $\phi$ is a formula, we will denote by $\phi^M$ the interpretation of $\phi$ in $M$. A \emph{sentence} is a formula with no free variables, and a \emph{theory} is a collection of conditions of the form $\phi = c$, where $\phi$ is a sentence and $c \in \RR$;  throughout the paper, we will consider only countable theories. A condition $\phi = c$ is \emph{satisfied} in a structure $M$ if $\phi^M = c$. A structure $M$ is a \emph{model} of the theory $T$, denoted by $M \models T$, if all conditions in $T$ are satisfied in $M$.

Fix a signature $L$. A \emph{fragment} of $\IL(L)$ is a countable collection $F \subseteq \IL(L)$ that contains all atomic formulas and is closed under finitary connectives, quantifiers, taking subformulas, and substitution of terms for variables. The smallest fragment is \emph{the finitary fragment} $\FL(L)$ that contains no infinitary formulas. If $F$ is a fragment and $T$ is a theory, we will say that $T$ is an $F$-theory if all sentences that appear in $T$ are in $F$.

Fix a fragment $F$. For an $F$-theory $T$, $M \models T$, and $\bar a \in M^{n}$, the \emph{type of $\bar a$} (or $F$-type if $F$ is not clear from the context)), denoted by $\tp(\bar a)$ (or $\tp_F(\bar a)$), is defined as a collection of all conditions of the form $\phi(x_1,\ldots,x_n)=c$ such that $\phi \in F$, and $\phi^M(\bar{a})=c$ (we write $\phi(\tp(\bar{a}))=c$). An \emph{$n$-type} of $T$ is the type of an $n$-tuple $\bar{a}$ in $M$ such that $M \models T$ (note that this definition agrees with the definition of a realizable type from Section 2.2 in \cite{HaMaTs}). The set $S_{n}(T)$ is the set of all $n$-types of $T$, i.e.,
\begin{equation*}
	S_{n}(T) = \{\tp(\bar a) : M \models T, \bar a \in M^{n}\}.
\end{equation*}

For  $\phi \in F$ and $r \in \QQ$, put
\[	[\phi < r]=\{p \in S_{n}(T) : \phi(p)=c \in p \mbox{ for some } c<r\}, \]
and define $[\phi \leq r]$, $[\phi=r]$ analogously. Observe that every set  $[\phi \leq r]$ can be written as some $[\psi=0]$, and every $[\phi<r]$ can be written as a countable union of some $[\psi=0]$. The \emph{logic topology} $\tau_{n}$ on $S_n(T)$ is given by pointwise convergence on formulas, i.e., basic open sets are of the form $[\phi<r]$. By \cite[Proposition 3.7]{HaMaTs}, this topology is Polish.

An important feature of type spaces in continuous logic is that, in addition to the logic topology, they are also equipped with a metric, which, in general, induces a finer topology. By \cite[Proposition 2.6]{HaMaTs}, it can be defined by
\begin{equation*}
	\partial(p, q) = \sup_{\phi \in F_1} |\phi(p) - \phi(q)|,
\end{equation*}
where $F_1$ is the collection of all $1$-Lipschitz formulas in $F$.

A type $p \in S_n(T)$ is \emph{isolated} if $\tau$- and $\partial$- topologies coincide on some neighborhood of $p$. A model $M$ of a theory $T$ is \emph{atomic} if all the types realized in $M$ are isolated. And it is \emph{$\aleph_0$-categorical} if it is a unique Polish metric structure modeling $\Th_F(M)$.

The space $\Mod(L)$ of all Polish metric structures in signature $L$ is defined as in \cite{BeDoNiTs}. Because functions can be easily coded as predicates (see Section 4 in \cite{BeDoNiTs}), we can assume that $L$ is a relational signature. Enumerate all predicates in $L$ as $P_0=d, P_1,\ldots$,  and let $n_0, n_1, \ldots$ be their
respective arities. Let $\Mod(L)$ be the set of all $\xi \in \prod_i \RR^{\NN^{n_i}}$ such that there exists a metric structure $M_\xi$ and a tail-dense sequence $(a_i)$ of elements of $M$ such that
\[ \xi(i)(j_0,\ldots, j_{n_i-1})=P^M_i(a_{j_0}, \ldots, a_{j_{n_i-1}})   \]
for all $i\in \NN$, $(j_0,\ldots, j_{n_i-1}) \in \NN^{n_i}$. Observe that $M_\xi$ can be obtained from $\xi$ by completing the pseudometric $P_0$ on $\NN$ (coded by $\xi(0) \in \RR^{\NN^2}$), extending predicates $P_i$ (coded by $\xi(i)$, $i>0$) to the completion, and taking the quotient with respect to the pseudometric. In other words, we can think of elements $\xi \in \Mod(L)$ as of Polish metric structures with a distinguished tail-dense sequence $(a_i)$. In particular, tuples in $M_\xi$ consisting of elements from this sequence can be unequivocally referred to as tuples from $\Seq$. Slightly abusing notation, we will often identify $\xi$ and $M_\xi$.

Beside the standard product topology on $\Mod(L)$, one can consider finer topologies generated by fragments, analogously to topologies generated by fragments in the setting of classical countable models (see Section 11 in \cite{Gao} for details). For a fragment $F$, a basis for the topology $t_F$ is given by sets of the form $$[\phi(\bar{a})<r]=\{M \in \Mod(L): \phi^M(\bar{a})<r\},$$ where $\phi \in F$, $\bar{a} \in \Seq$, and $r \in \QQ$. Note that the standard topology can be regarded as the topology generated by sets $[\phi(\bar{a}) < r]$, where $\phi$ is finitary and quantifier-free. 

For a theory $T$, the space $\Mod(T) \subseteq \Mod(L)$ is the space of all Polish metric structures modeling $T$. By \cite[Lemma 3.5]{HaMaTs}, topologies $t_F$ are Polish on $\Mod(T)$, provided that $F$ contains all sentences in $T$. The symbol $\cong_T$ denotes the isomorphism relation on $\Mod(T)$, and $[M]$ is the isomorphism class of $M \in \Mod(T)$. We say that $\cong_T$ is potentially $\bPi^0_\alpha$ (where $\bPi^0_\alpha$ refers to Borel sets of multiplicative rank $\alpha$) if there is a Polish topology $t$ on $\Mod(T)$, consisting of Borel sets in the standard product topology on $\Mod(L)$, such that $\cong_T \in \bPi^0_\alpha(t \times t)$.

For a metric space $(X,d)$, denote balls in $X$ by
\[ B^X_r(a)=\{b \in X: d(a,b)<r \}, \, B^X_{\leq r}(a)=\{b \in X: d(a,b) \leq r \}; \]
if $X$ is clear from the context, we will write $B_r(a)$ and $B_{\leq r}(a)$. We will also consider balls in $X^\NN$ and $X^{<\omega}$ around finite tuples, using the extension of $d$ defined above; in particular $B^{X^\NN}_r(\emptyset)=X^\NN$ and $B^{X^{<\omega}}_r(\emptyset)=X^{<\omega}$.

For a Polish metric structure $M \in \Mod(L)$, let $$D(M)=\{ (y_i) \in M^\NN: \{y_i\} \mbox{ is tail-dense in } M  \}.$$ $D(M)$ is clearly a $G_\delta$ set in $M^\NN$, and therefore a Polish space. Denote by $\pi \colon D(M) \to [M]$ the map given by
\begin{equation}
	\label{eq:def-pi}
	P_i^{\pi(y)}(\bar{a}) = P_i^{M}(y(a_0), \ldots, y(a_{n_i-1}))
\end{equation}
for $y \in D(M)$, $\bar{a} \in \NN^{n_i}$. It is surjective, and continuous with respect to topologies generated by fragments. More importantly, $\pi$ plays the role of the logic action in the context of countable structures. For $A \subseteq \Mod(L)$, $\bar{a} \in \Seq$, and $u \in \QQ^+$, define $A^{* \bar{a},u}$ by
\[ M \in A^{* \bar{a},u} \Leftrightarrow  \forall^* y \in B^{D(M)}_{u}(\bar{a}) (\pi(y) \in A), \]
and $A^*=A^{* \emptyset,0}$; $A^{\Delta \bar{a},u}$ and $A^\Delta$ are defined similarly. The operations $A^{* \bar{a},u}$ are analogs of Vaught transforms. By \cite[Theorem 6.3]{BeDoNiTs}, which is a continuous counterpart of the Lopez-Escobar theorem, $[M]$ is Borel for $M \in \Mod(L)$, and every isomorphism-invariant Borel $A \subseteq \Mod(L)$ is of the form $\Mod(T)$ for some theory $T$.

\section{AE families}

Using Vaught transforms, one can characterize isomorphism classes of countable structures in terms of satisfiability of appropriately chosen formulas. For example, if $[M]$ is $\bPi^0_2$, there are formulas $\phi_{k,l}$, $k,l \in \NN$, such that  $N \in [M]$ iff
\[  \forall  \bar{b} \forall k  \exists \bar{c} \supseteq \bar{b} \exists l  (N \models \phi_{k,l}(\bar{c}))  \]
(see the proof of \cite[Theorem 11.5.7]{Gao} for details). It turns out that this approach can be generalized to higher Borel complexity, and to the continuous setting. As it will turn out in the next section, it allows for taking advantage of type spaces in defining nicely behaving invariants of isomorphism for classes of locally compact structures.  

For a fixed (countable) fragment $F$ in signature $L$, $\alpha<\omega_1$, and a tuple $\bar{x}$ of free variables, an \emph{$\alpha$-AE family} $P(\bar{x})$ is defined as follows. An $(-1)$-family $P(\bar{x})$ is a formula $\phi(\bar{x})$ in $F$. Provided that $\gamma$-AE families have been defined for $\gamma<\beta$, where $\beta=0$ or $\beta$ is a limit ordinal, a $\beta$-AE family $P(\bar{x})$ is a collection of $\gamma$-AE families $p_k(\bar{x})$, $k \in \NN$, $\gamma<\beta$, a $(\beta+1)$-AE family $P(\bar{x})$ is a collection of $\gamma$-AE families $p_{k,l}(\bar{x}_{k,l})$, $\gamma<\beta$, $k,l \in \NN$, $\bar{x} \subseteq \bar{x}_{k,l}$, and a $(\beta+n)$-AE family $P(\bar{x})$, $2 \leq n < \omega$, is a collection of $(\beta+n-2)$-AE families $p_{k,l}(\bar{x}_{k,l})$, $k,l \in \NN$, $\bar{x} \subseteq \bar{x}_{k,l}$. Moreover, every $\alpha$-AE family $P(\bar{x})=\{p_{k,l}(\bar{x}_{k,l})\}$, $\alpha \geq 1$, comes equipped with a fixed $u_P \geq 0$ such that $u_P \geq u_{p_{k,l}}$, $k,l \in \NN$.

We say that a tuple $\bar{a}$ in $M \in \Mod(L)$ realizes a $(-1)$-AE family $P(\bar{x})=\phi(\bar{a})$ if $\phi^M(\bar{a})=0$, and $\bar{a}$ realizes a $\beta$-AE family $P(\bar{x})$, where $\beta=0$ or $\beta$ is a limit ordinal, if it realizes every $p(\bar{x}) \in P(\bar{x})$. Finally, $\bar{a}$ realizes a $(\beta+n)$-AE family $P(\bar{x})=\{p_{k,l}(\bar{x}_{k,l})\}$, $1 \leq n< \omega$, if it holds in $M$ that 
%it
\[ \forall \bar{b} \in B^{M^{<\omega}}_{u_P}(\bar{a}) \forall v>0 \forall k \exists \bar{c} \in B^{M^{<\omega}}_{v}(\bar{b}) \exists l  \, (\bar{c} \mbox{ realizes } p_{k,l}(\bar{x}_{k,l}) \mbox{ in } M ). \]
If $\emptyset$ in $M$ realizes $P(\emptyset)$, we say that $M$ models $P$.

\begin{remark}
\label{re:Smallv}
Note that in order to verify that $\bar{a}$ realizes $P(\bar{x})$, it suffices to check that the above condition holds for $\bar{b}, \bar{c} \in \Seq$ with $|\bar{b}|\geq |\bar{a}|$, $|\bar{c}|\geq |\bar{b}|$, and for all sufficiently small $v>0$. 
\end{remark}

\begin{theorem}
\label{th:AEfamilies}
Let $F$ be fragment in signature $L$, and let $1 \leq \alpha<\omega_1$. Suppose that $A \in \bPi^0_{1+\alpha}(t_F)$ for some $A \subseteq \Mod(L)$. For every $\bar{a} \in \Seq$, and $u \in \QQ^+$, there exists an $\alpha$-AE family $P(\bar{x})$ such that
\[ A^{*\bar{a},u }=\{ N \in \Mod(L): \bar{a} \mbox{ realizes } P(\bar{x}) \mbox{ in } N \}. \]
\end{theorem}

\begin{proof}
	We prove the theorem by induction on $\alpha$. Let us consider the base cases $\alpha=1$, and $\alpha=2$. Fix  $A \subseteq \Mod(L)$ such that $A \in \bPi^0_{2}(t_F)$, $\bar{a} \in \Seq$, $u \in \QQ^+$, and closed sets $A_{k,l}$, $k,l \in \NN$, in $t_F$ of the form $[\phi=0]$ such that
	\[ A=\bigcap_k \bigcup_l A_{k,l}.\]
	Then
	\[ N \in A^{*\bar{a}, u} \Leftrightarrow \forall^* y \in B^{D(N)}_{u}(\bar{a})  \forall k \exists l \, (\pi(y) \in A_{k,l}) \Leftrightarrow \]
	\[   \forall \bar{b} \in B^{N^{<\omega}}_{u}(\bar{a}) \forall v>0 \forall k  \exists^* y \in B^{D(N)}_{v}(\bar{b}) \exists l  \, (\pi(y) \in A_{k,l}) \Leftrightarrow  \]
	\[   \forall \bar{b} \in B^{N^{<\omega}}_{u}(\bar{a}) \forall v>0  \forall k  \exists \bar{c} \in B^{N^{<\omega}}_{v}(\bar{b}) \exists l \exists w>0 \forall^* y \in B^{D(N)}_{w}(\bar{c}) \, (\pi(\bar{y}) \in A_{k,l}) \Leftrightarrow  \]
	\[   \forall \bar{b} \in B^{N^{<\omega}}_{u}(\bar{a}) \forall v>0  \forall k  \exists \bar{c} \in B^{N^{<\omega}}_{v}(\bar{b}) \exists l \exists w>0 \, (N \in A^{* \bar{c}, w}_{k,l}). \]
	Since $A_{k,l}$ are closed, we have that
	%
	%\[ \forall k \forall \bar{b}   \exists l \exists c \in [\bar{a}\bar{b}] \, (g.N \in A_{k,l}), \]
	%
	\[ N \in A^{*\bar{c}, w}_{k,l} \Leftrightarrow \neg \exists^* y \in B^{D(N)}_{w}(\bar{c}) \, ( \pi(y) \not \in A_{k,l}) \Leftrightarrow  \]
	\[ \neg \exists y \in B^{D(N)}_{w}(\bar{c}) \, ( \pi(y) \not \in A_{k,l}) \Leftrightarrow \neg \exists \bar{d} \, (d^N(\bar{c},\bar{d})<w \mbox{ and } \bar{d} \not \in A_{k,l}).  \]
	Moreover, $A_{k,l}$ are of the form $[\phi=0]$, so there are formulas $\phi_{k,l}(\bar{x}_{k,l})$ in $F$ such that
	\[ N \in A^{*\bar{a}, u} \Leftrightarrow \forall \bar{b} \in B^{N^{<\omega}}_{u}(\bar{a}) \forall v>0  \forall k  \exists \bar{c} \in B^{N^{<\omega}}_v(\bar{b}) \exists l \, (\phi^N_{k,l}(\bar{c})=0) \]
	Put $P(\bar{x})=\{\phi_{k,l}(\bar{x}_{k,l})\}$, and $u_P=u$. It is a $1$-AE family, and
	\[ N \in A^{*\bar{a}, u}  \Leftrightarrow \bar{a} \mbox{ realizes } P(\bar{x}) \mbox{ in } N. \]

	For $\alpha=2$, fix $A \in \bPi^0_3(t_F)$, a tuple $\bar{a}$, and closed in $t_F$ sets $A_{k,l,m}$ of the form $[\phi=0]$  such that
	\[ A=\bigcap_k \bigcup_l \bigcap_m A_{k,l,m}.\]
	Then
	\[ N \in A^{*\bar{a}, u } \Leftrightarrow \forall^* y \in B^{D(N)}_{u}(\bar{a}) g \forall k \exists l \forall m \, (\pi(y) \in A_{k,l,m}) \Leftrightarrow \]
	\[ \forall \bar{b} \in B^{N^{<\omega}}_{u}(\bar{a}) \forall v>0 \forall k \exists^* y \in B^{D(N)}_{v}(\bar{b}) \exists l \forall m  \, (\pi(y) \in A_{k,l,m}) \Leftrightarrow  \]
	%\[ \forall \bar{b} \in B^N_{u}(\bar{a}) \forall v>0 \forall k   \exists l \exists^* y \in B^N_{v}(\bar{b}) \forall m  \, (\pi(y) \in A_{k,l,m}) \Leftrightarrow \]
	\[ \forall \bar{b} \in B^{N^{<\omega}}_{u}(\bar{a}) \forall v>0 \forall k \exists \bar{c} \in B^{N^{<\omega}}_{v}(\bar{b}) \exists w>0 \exists l \forall^* y \in B^{D(N)}_{w}(\bar{c})  \forall m  \, (\pi(y) \in A_{k,l,m}).  \]
	%\[  \forall k \forall \bar{b} \in B^N_{u}(\bar{a})  \forall v>0 \exists \bar{c} \in B^N_{v}(\bar{b})  \exists l \exists w>0 \forall^* y \in B^N_{w}(\bar{c}) \, (\pi(y) \in A_{k,l,m}).  \]
	%
	Since $A_{k,l,m}$ are closed, the condition $\forall^* y \in B^{D(N)}_{w}(\bar{c}) \forall m \, (\pi(y) \in A_{k,l,m})$ is also closed for every $w$, $\bar{c}$, $k$, and $l$. Therefore %the above is equivalent to 
	%
	%\[ \forall k \forall \bar{b}   \exists l \exists \bar{c} \forall m \, \neg(\exists g \in [\bar{a}\bar{b}\bar{c}]   \, (g.N \in A_{k,l,m})),  \]
	%
	there are formulas $\phi_{k,l,m}(\bar{x}_{k,l,m})$ in $F$ such that the above is equivalent to
	\[ \forall \bar{b} \in B^{N^{<\omega}}_{u}(\bar{a})  \forall v>0 \forall k \exists \bar{c} \in B^{N^{<\omega}}_{v}(\bar{b}) \exists l \forall m \, (\phi^N_{k,l,m}(\bar{c})=0). \]
	In other words, for $p_{k,l}(\bar{x}_{k,l})=\{ \phi_{k,l,m}(\bar{x}_{k,l}): m \in \NN \}$, and the $2$-AE family $P(\bar{x})=\{p_{k,l}(\bar{x}_{k,l})\}$, $u_P=u$, we have that 
	\[ N \in A^{*\bar{a}, u} \Leftrightarrow \bar{a} \mbox{ realizes } P(\bar{x}) \mbox{ in } N. \]
	
	Suppose $\alpha>2$, and write $\alpha=\beta+n$, where $\beta$ is $0$ or a limit ordinal, and $n<\omega$. For $n=0$, this is straightforward. For $n=1$ or $n>2$, exactly the same argument as for $\alpha=1$ works, only we use the inductive assumption to deal with $A^{*\bar{c}, w}_{k,l}$. And for $n=2$, we repeat the argument for $\alpha=2$, again using the inductive assumption.
\end{proof}

In particular, since $[M]=[M]^*$, we get

\begin{corollary}
\label{co:IsoAEAlpha}
Let $F$ be fragment in signature $L$, and let $1 \leq \alpha<\omega_1$. Suppose that $[M] \in \bPi^0_{1+\alpha}(t_F)$ for some $M \in \Mod(L)$. There exists an $\alpha$-AE family $P_M$ such that
\[  [M]= \{ N \in \Mod(L) : N \mbox{ models } P_M \}. \]
\end{corollary}

\section{Locally compact structures}
AE families are hereditarily countable objects that characterize isomorphism classes of Polish metric structures. Unfortunately (although unsurprisingly), it is not always possible to assign them to structures in a definable (i.e., Borel) and isomorphism-invariant manner. However, the case of locally compact structures is simpler. Let us start with basic facts about type spaces of theories with locally compact models.

For a fragment $F$, $F$-theory $T$, locally compact $M \in \Mod(T)$, $n \in \NN$, and $n$-tuple $\bar{a}$ in $M$, let 
\[ \Theta_n(M)=\{\tp_F(\bar{b}): \bar{b} \in M^n\},\]
\[	\rho_M(\bar{a}) = \sup \{r \in \RR : \cl{B^{M^n}_r(\bar{a})} \text{ is compact} \},\]
or simply $\rho(\bar{a})$, when $M$ is clear from the context.

\begin{lemma}[Lemma 6.2 in \cite{HaMaTs}]
	\label{l:loc-cpct-type-space}
	Let $\Phi \colon (M^n, d) \to (S_n(T), \partial)$ be defined by $\Phi(\bar a) = \tp (\bar a)$. Then the following hold:
	\begin{enumerate}
		\item \label{i:p:loc-cpct-type-space:1} $\Phi$ is a contraction for the metrics $d$ on $M^n$ and $\partial$ on $S_n(T)$.
		\item \label{i:p:loc-cpct-type-space:2} If $r < \rho(\bar a)$, then $\Phi(B_{\leq r}(\bar a)) = B_{\leq r}(\Phi(a))$. In particular, \\ $B_{\leq r}(\tp (\bar a)) \subseteq \Theta_n(M)$, $B_{\leq r}(\tp (\bar a))$ is $\partial$-compact, and $\tau_n$- and $\partial$-topology coincide on $B_{\leq r}(\tp (\bar a))$.
		\item \label{i:p:loc-cpct-type-space:3} If $r \leq \rho(\bar a)$, then $\Phi(B_r(\bar a)) = B_r(\Phi(\bar a))$. In particular, $\Phi$ is an open mapping.
		\item \label{i:p:loc-cpct-type-space:4} The set $\Theta_n(M)$ is open in $(S_n(T), \partial)$ and the space $(\Theta_n(M), \partial)$ is locally compact and separable. %In particular, it is closed in $(\tS_n(F), \dtp)$.
	\end{enumerate}
\end{lemma}

For each topology $\tau_n$ fix a countable basis $\UUU_n=\{U_{l,n}\}$ containing $\emptyset$ and the whole space, and put $\UUU=\bigcup_n \UUU_n$. For $U \in \UUU_n$, $\epsilon>0$, and $n$-tuple $\bar{a}$ in $M$, we say that $(U,\epsilon)$ is \emph{$\bar{a}$-good} in $M$ if
\begin{itemize}
	\item $\tp(\bar{a}) \in U$,
	\item $2\epsilon<\rho(\bar{a})$,
	\item there is $\delta>0$ such that $U \cap B_{2\epsilon}(\tp(\bar{a})) \subseteq B_{\epsilon-\delta}(\tp(\bar{a}))$. 
\end{itemize}

\begin{remark}
\label{re:good}
%Fix a theory $T$, $M \in \Mod(T)$, and a tuple $\bar{a}$ in $M$. Suppose that $M$ is locally compact.
The following observations easily follow from the fact that $\partial$- and $\tau$- topologies coincide on compact subsets of $(S_n(T),\partial)$.
\begin{enumerate}
\item For every $\delta>0$ there exist $U \in \UUU$ and $0<\epsilon<\delta$ such that $(U,\epsilon)$ is $\bar{a}$-good,
\item if $(U,\epsilon)$ is $\bar{a}$-good, then $$\cl[\tau]{B_\epsilon(\tp(\bar{a})) \cap U } \subseteq \Theta_{|\bar{a}|}(M),$$
\item if $(U,\epsilon)$ is $\bar{a}$-good, there is $\delta>0$ such that $d(\bar{a},\bar{a}')<\delta$ implies that  $(U,\epsilon)$ is $\bar{a}'$-good, and $$U \cap B_{\epsilon}(\tp(\bar{a}))=U \cap B_{\epsilon}(\tp(\bar{a}')).$$
\end{enumerate}
\end{remark}
Now, for $\bar{a} \in \Seq$, $U \in \UUU_n$, and $\epsilon \in \QQ^+$, define
\[ T^0_{U,\epsilon}(\bar{a})=\cl[\tau]{B_\epsilon(\tp(\bar{a})) \cap U }, \]
if $(U,\epsilon)$ is $\bar{a}$-good,
\[ T^0_{U,\epsilon}(\bar{a})=\emptyset, \]
otherwise, and
\begin{multline}
T^{\alpha}_{U,\epsilon}(\bar{a})=\{ T^\beta_{U',\epsilon'}(\bar{a}'): \beta<\alpha, |\bar{a}'| \geq |\bar{a}|, U' \in \UUU_{|\bar{a}'|}, U' \upharpoonright |\bar{a}| \subseteq U,  \epsilon' \leq \epsilon \} %, \\  (U',\epsilon') \mbox{ is } \bar{a}'\mbox{-good} \}	
\end{multline}
for $\alpha>0$. Also, for $u>0$, put
\begin{multline}
T^\alpha_u(\bar{a})=\{ T^\beta_{U,v}(\bar{b}): \beta<\alpha, \bar{b} \in B^{M^{<\omega}}_u(\bar{a}), |\bar{b}| \geq |\bar{a}|, U \in \UUU_{|\bar{b}|}, v>0 \}, %, \\  (U,v) \mbox{ is } \bar{b}\mbox{-good} \}, 
\end{multline}

\[ T^\alpha(M)=T^\alpha_1(\emptyset).\]
As tuples in the definition of $T^{\alpha}_{U,\epsilon}(\bar{a})$ range over $\Seq$, $U$ range over a countable family, and $\epsilon \in \QQ^+$, $T^1(M)$ is a countable family of $\tau$-closed sets, $T^2(M)$ is a countable family of countable families of $\tau$-closed sets, etc. Moreover, using Remark \ref{re:good}(3), it is straightforward to observe that

\begin{remark}
\label{re:CongSameT}
$M \cong N$ implies that $T^\alpha(M)=T^\alpha(N)$ for all $\alpha \geq 1$.
\end{remark}

\begin{proposition}
\label{pr:AET}
Let $F$ be a fragment, and let $T$ be an $F$-theory. Suppose that $M,N \in Mod(T)$ are locally compact, and  $T^\alpha_{u}(\bar{a})=T^\alpha_{u'}(\bar{a}')$ for some tuples $\bar{a}$, $\bar{a}'$ in $M$, $N$, respectively. Then every $\alpha$-AE family $P(\bar{x})$ with $u_P \leq u$ realized by $\bar{a}'$, is also realized by $\bar{a}$. 
\end{proposition}

\begin{proof}
Suppose that $T^1_{u}(\bar{a})=T^1_{u'}(\bar{a}')$, and fix a $1$-AE family $P(\bar{x})=\{p_{k,l}(x_{k,l})\}$ realized by $\bar{a}'$, and with $u_P \leq u$. Fix $\bar{b} \in B^{M^{<\omega}}_{u_P}(\bar{a})$, $v>0$, and $k \in \NN$. By Remarks \ref{re:Smallv} and \ref{re:good}(1), we can assume that $v<\rho(\bar{b})$, and there is $U \in \UUU$ such that $(U,v/2)$ is $\bar{b}$-good. Find $\bar{b}' \in  B^{N^{<\omega}}_{u_P}(\bar{a}')$, $U' \in \UUU$, and $v'>0$ such that $T^0_{U,v/2}(\bar{b})=T^0_{U',v'}(\bar{b}')$. As $\bar{a}'$ realizes $P(\bar{x})$, there is $\bar{c}' \in B^{N^{<\omega}}_{v/2}(\bar{b}')$, and $l$ such that $p_{k,l}^N(\bar{c}')=0$. But then 
\[ \inf_{\bar{x}} [(d(\bar{b},\bar{x}) \dotdiv v/2) \vee p_{k,l}(\bar{x})] \in \tp(\bar{b}'), \]
and, by Remark \ref{re:good}(2), there is $\bar{d} \in  B^N_{v/2}(\bar{b})$ with $\tp(\bar{d})=\tp(\bar{b}')$. By the compactness of $B^M_{v/2}(\bar{d})$, there is $\bar{c} \in B^M_{v/2}(\bar{d})$ such that $p_{k,l}^M(\bar{c})=0$. Clearly, $\bar{c} \in B^{M^{<\omega}}_v(\bar{b})$. As $\bar{b}$, $v$ and $k$ were arbitrary, this shows that $\bar{a}$ realizes $P(\bar{x})$.

Suppose now that $T^2_{u}(\bar{a})=T^2_{u'}(\bar{a}')$, and let $P(\bar{x})$ be a $2$-AE family realized by $\bar{a}'$, and with $u_P \leq u$. Fix $\bar{b} \in  B^{M^{<\omega}}_{u_P}(\bar{a})$, $v>0$, $k \in \NN$, and $U \in \UUU$ such that $(U,v/2)$ is $\bar{b}$-good. Fix $\bar{b}' \in B^{N^{<\omega}}_{u_P}(\bar{a}')$, $U' \in \UUU$ and $v'>0$ such that $T^1_{U,v/2}(\bar{b})=T^1_{U',v'}(\bar{b}')$. As $\bar{a}'$ realizes $P(\bar{x})$, there is $l$, and $\bar{c}' \in B^{N^{<\omega}}_{v/2}(\bar{b}')$ such that $\tp(\bar{c}') \upharpoonright |\bar{b}'| \in U'$, and  $\bar{c}'$ realizes $p_{k,l}(\bar{x}_{k,l})$. Fix $V,V' \in \UUU$, $0<w,w' \leq v/2$, and $\bar{d} \in B^{M^{<\omega}}_{v/2}(\bar{b})$ such that $(V',w')$ is $\bar{c}'$-good, and $T^0_{V,w}(\bar{d})=T^0_{V',w'}(\bar{c}')$. Then there is $\bar{c} \in B^{M^{<\omega}}_{w}(\bar{d})$ with $\tp(\bar{c})=\tp(\bar{c'})$, i.e., $\bar{c} \in B^{M^{<\omega}}_v(\bar{b})$, and $\bar{c}$ realizes $p_{k,l}(\bar{x}_{k,l})$.
	
For $\alpha>1$, this is an easy induction.
\end{proof}

Proposition \ref{pr:AET} combined with Corollary \ref{co:IsoAEAlpha} immediately gives:

\begin{theorem}
\label{th:isoTalpha}
Let $F$ be a fragment, and let $T$ be an $F$-theory all of whose models are locally compact. Suppose that $[M] \in \Pi^0_{1+\alpha}(t_F)$, $\alpha \geq 1$, for some $M \in \Mod(T)$. Then $$[M]=\{N \in \Mod(T): T^{\alpha}(N)=T^{\alpha}(M)\}.$$
\end{theorem}

\begin{theorem}
\label{th:CtbleModels}
Let $F$ be a fragment, and let $T$ be an $F$-theory all of whose models are locally compact. Then $\cong_T$ is classifiable by countable structures.
\end{theorem}

\begin{proof}
First, for a given $M \in \Mod(T)$, we construct a countable structure $C_M$, essentially, as in the proof of \cite{Hj}[Lemma 6.30]. Its universe consists of elements $x$ of the form $$x=(\cl[\tau]{B_\epsilon(\tp(\bar{a})) \cap U },|\bar{a}|,U,\epsilon),$$ where $\bar{a} \in \NN^{<\NN}$, $U \in \UUU_{|\bar{a}|}$, $\epsilon \in \QQ^+$, and $(U,\epsilon)$ is $\bar{a}$-good. The relevant information carried by these objects is recorded with an aid of the relations $O_l$, $R_{k,l,\delta}$, $k,l \in \NN$, $\delta \in \QQ^+$, and $E$, defined, for $x=(\cl[\tau]{B_\epsilon(\tp(\bar{a})) \cap U },|\bar{a}|,U,\epsilon)$, $x'=(\cl[\tau]{B_{\epsilon'}(\tp(\bar{a}')) \cap U' },|\bar{a}'|,U',\epsilon')$, as follows:

\begin{itemize}
	\item $O_l(x)$ iff $U_{l,|\bar{a}|} \cap \cl[\tau]{B_\epsilon(\tp(\bar{a})) \cap U }=\emptyset$,
	\item $R_{k,l,\delta}(x)$ iff $k=|\bar{a}|$, $U=U_{l,k}$, $\delta=\epsilon$,
	\item $x E x'$ iff $|\bar{a}'| \geq |\bar{a}|$, $U' \upharpoonright |\bar{a}| \subseteq U$,  $\epsilon' \leq \epsilon$.
\end{itemize}

By Remark \ref{re:good}(3), $M \cong N$ implies that $C_M=C_N$. On the other hand, as the relations $O_l$ record complements of sets $\cl[\tau]{B_\epsilon(\tp(\bar{a})) \cap U }$ in $S_{|\bar{a}|}(T)$, we have that $$\pi((\cl[\tau]{B_\epsilon(\tp(\bar{a})) \cap U },|\bar{a}|,U,\epsilon))=(\cl[\tau]{B_\epsilon(\tp(\bar{a})) \cap U },|\bar{a}|,U,\epsilon)$$ for any isomorphim $\pi:C_M \rightarrow C_N$.
It obviously follows that $T^0_{U,\epsilon}(\bar{a})=T^0_{U,\epsilon}(\bar{a'})$, if $\pi(x)=x'$. And $E$ warranties that $T^\alpha_{U,\epsilon}(\bar{a})=T^\alpha_{U,\epsilon}(\bar{a'})$ for $\alpha>0$. In particular, $T^\alpha(M)=T^\alpha(N)$ for $\alpha<\omega_1$.

It is not hard to construct a Borel mapping $\Mod(T) \rightarrow 2^\NN$, $M \mapsto D_M$, so that $D_M$ codes a countable model isomorphic to $C_M$. First, by \cite[Lemma 6.4]{HaMaTs}, the mappings
\[ \Mod(T) \times \NN^n \rightarrow \RR, \ (M,\bar{a}) \mapsto \rho_M(\bar{a}), \]
\[ \Mod(T) \times \NN^n \times \QQ^+ \rightarrow \mathcal{K}(S_n(T)), (M,\bar{a},r) \mapsto B_{\leq r}(\tp(\bar{a})), \]
where $\mathcal{K}(X)$ is the standard Borel space of closed subsets of $X$, are Borel. Therefore the relation  ``$(U,\epsilon)$ is $\bar{a}$-good in $M$'', regarded as a subset of $\Mod(T) \times \NN^{<\NN} \times \UUU \times \QQ^+$, is also Borel. This gives rise to a Borel enumeration $e: \NN \rightarrow (\bigsqcup_n \KK(S_n(T))) \times \UUU \times \QQ^+$ of the universe of $C_M$. Using this $e$, we can easily construct the required Borel mapping $M \mapsto D_M$. 

By \cite[Corollary 5.6]{BeDoNiTs}, for every $M \in \Mod(T)$, the isomorphism class $[M]$ is Borel, i.e., $[M] \in \bPi^0_\alpha(t_F)$ for some $\alpha<\omega_1$. Hence, Theorem \ref{th:isoTalpha} implies that $M \cong N$ iff $D_M \cong D_N$. 
% In other words, the isomorphism relation $\cong_T$ is classifiable by countable models, and thus Borel reducible to graph isomorphism.
\end{proof}

The next result confirms a conjecture stated by Gao and Kechris in \cite{GaKe} (Hjorth, see \cite{GaKe}, announced a positive answer for its weaker form, with a $\Delta^1_2$-reduction).

\begin{theorem}
\label{th:Isometry}
Isometry of locally compact Polish metric spaces is Borel reducible to graph isomorphism.
\end{theorem}

\begin{proof}
Every locally compact Polish metric space $K$ can be coded as $M_K \in \Mod(L)$ with the trivial signature $L$, and metric bounded by $1$. Simply, pick a countable tail-dense subset of $K$, and replace the original metric $d$ with the metric $1/(1+d)$ which does not change the isometry relation. Actually, for $\mathcal{LC} \subseteq \mathcal{K}(\UU)$ denoting the standard Borel space of all locally compact Polish metric spaces, regarded as subsets of the Urysohn space $\UU$, the coding $\mathcal{LC} \rightarrow \Mod(L)$, $K \mapsto M_K$, can be defined in a Borel way: the Kuratowski--Ryll-Nardzewski theorem yields a Borel function $f:\mathcal{K}(\UU) \rightarrow \UU^\NN$ such that $f(K)=(k_n)$ is a tail-dense subset of $K$. As the signature $L$ is trivial, the isomorphism relation $\cong$ is just the isometry relation. Moreover, the property of being locally compact can be expressed as a sentence in $\IL(L)$, so the set of all possible codes of locally compact Polish metric spaces is of the form $\Mod(T)$. By Theorem \ref{th:CtbleModels} and \cite[Theorem 13.1.2]{Gao}, the isometry relation is Borel reducible to graph isomorphism.
\end{proof}

\paragraph{\textbf{Borel isomorphism relations}}
As a matter of fact, a more detailed analysis can be performed in case the isomorphism relation is Borel. Let $\PP^0(\NN)=\NN$, and, for $\alpha <\omega_1$, let $\PP^\alpha(\NN)=$ all countable subsets of $\PP^{<\alpha}(\NN) \cup \NN$, where $\PP^{<\alpha}(\NN)= \bigcup_{\beta<\alpha} \PP^{\beta}(\NN)$. Thus, $\PP^1(\NN)$ (the reals) consists of all subsets of $\NN$, $\PP^2(\NN)$ of all countable sets of reals, etc. We denote the equality relation on $\PP^{\alpha}(\NN)$ by $=_{\alpha}$. In \cite{HjKeLo}, it is explained how elements of $\PP^{\alpha}(\NN)$ can be coded as countable models so that $=_{\alpha}$ becomes an isomorphism relation of Borel class $\bPi^0_{\alpha+1}$.

For a fragment $F$ in signature $L$, the rank $\rk_F(\phi)$ is defined by $\rk_F(\phi)=0$ for $\phi \in F$, $\rk_F(\sup_x\phi)=\rk_F(\inf_x\phi)=\rk_F(r\phi)=\rk_F(\phi)$, $\rk_F(\phi_1 \vee \phi_2)=\rk_F(\phi_1 + \phi_2)=\max\{\rk_F(\phi_1), \rk_F(\phi_2)\}$, and $\psi$ $\rk_F(\phi)=\sup\{\rk_F(\phi_i)+1\}$ if  $\phi$ is an infinite supremum $\bigvee_i \phi_i$ or infinite infimum $\bigwedge_i \phi_i$. We also put $\rk_F(X)=\sup \{\rk_F(\phi): \phi \in X\}$ for a collection of formulas $X$. Note that it is straightforward to code a formula $\phi$ as an element of $\PP^\alpha(\NN)$ if $\rk_F(\phi)\leq \alpha$.

\begin{theorem}
	\label{th:Sigma0alpha}
	Let $L$ be a signature, let $t$ be a Polish topology on $\Mod(L)$ consisting of Borel subsets of the standard topology, and let $\alpha<\omega_1$. There exists a fragment $F$ such that $A^*, A^{* \bar{a},1/k} \in \bPi^0_\alpha(t_{F})$ for every $A \in \bPi^0_\alpha(t)$, $\bar{a} \in \NN^{<\NN}$, and $k>0$.
\end{theorem}

\begin{proof}
	We prove by induction on $\alpha$ that if $A \in \bSigma^0_\alpha(t)$, then there is a fragment $F$ such that $A^\Delta, A^{\Delta u,k} \in \bSigma^0_\alpha(t_F)$. For $\alpha=2$, fix a countable basis $\mathcal{A}$ for the topology $t$. Without loss of generality, we can assume that it contains the standard basis on $\Mod(L)$, and is closed under finite intersections. It is straightforward to observe that for every $M \in \Mod(L)$, $A \subseteq \Mod(L)$, $\bar{a} \in \NN^{<\NN}$, and $k>0$
	\[ \forall^* y \in B^{D(M)}_{1/k}(\bar{a}) (\pi(y) \in A) \Leftrightarrow \mbox{inf}^*_{y \in D(M)} (\chi_{A}(\pi(y)) \vee kd(y,\bar{a}))=1. \]
	By \cite[Theorem 6.3]{BeDoNiTs}, there exists a formula $\phi_{A,k}$ such that
	\[  \forall^* y \in B^{D(M)}_{1/k}(\bar{a}) (\pi(y) \in A) \Leftrightarrow \phi_{A,k}^M(\bar{a})=1.  \]
	Let $F$ be the fragment generated by such formulas for $A \subseteq \Mod(L)$ whose complement is in $\mathcal{A}$.
	Fix $A \in \bSigma^0_2(t)$, and $A_{n,m}$, $n, m \in \NN$, whose complement is in $\mathcal{A}$ and
	\[ A=\bigcup_n \bigcap_m A_{n,m}. \]
	%
	%where $A_{n,m} \in \mathcal{A}$.
	Then
	\[ M \in A^{\Delta \bar{a},1/k} \Leftrightarrow  \exists^* y \in B^{D(M)}_{1/k}(\bar{a}) \, \exists n \forall m \, (\pi(y)\in A_{n,m}) \Leftrightarrow \]
	\[ \exists n \exists^* y \in B^{D(M)}_{1/k}(\bar{a}) \forall m \, (\pi(y) \in A_{n,m}) \Leftrightarrow \]
	\[ \exists n \exists \bar{a}'  \in \NN^{<\NN} \exists k' \forall^* y \in B^{D(M)}_{1/k'}(\bar{a}') \forall m \, (d(\bar{a}',\bar{a}) \leq 1/k \mbox{ and } \pi(y) \in A_{n,m}) \Leftrightarrow \]
	\[ \exists n \exists \bar{a}' \in \NN^{<\NN}  \exists k'  \forall m \forall^* y \in B^{D(M)}_{1/k'}(\bar{a}') \, (d(\bar{a}',\bar{a}) \leq 1/k \mbox{ and } \pi(y) \in A_{n,m}), \]

	so there exist $B_{\bar{a}',k',n,m} \in \mathcal{A}$ such that 
	\[  A^{\Delta \bar{a},1/k}=\bigcup_n \bigcup_{\bar{a}'} \bigcup_{k'} \bigcap_m  \bigcap_{\epsilon>0} [\phi_{B_{\bar{a}',k',n,m},k}(\bar{a}) \leq \epsilon], \]
	i.e., $A \in \bSigma^0_2(t_F)$. For $A^\Delta$, the argument is analogous. 
	
	Suppose now that the lemma holds for all $\beta<\alpha$. Fix $A \in \bSigma^0_\alpha(t)$, $A_{n,m} \in \bSigma^0_{\beta_{n,m}}(t)$, $n,m \in \NN$, $\beta_{n,m}<\alpha$, such that
	\[ A=\bigcup_n \bigcap_m A_{n,m}. \]
	Then
	\[ M \in A^{\Delta \bar{a},1/k} \Leftrightarrow  \exists^* y \in B^{D(M)}_{1/k}(\bar{a}) \, \exists n \forall m \, (\pi(y)\in A_{n,m}) \Leftrightarrow \]
	\[ \exists n \exists^* y \in B^{D(M)}_{1/k}(\bar{a}) \forall m \, (\pi(y) \in A_{n,m}) \Leftrightarrow \]
	\[ \exists n \exists \bar{a}' \in \NN^{<\NN} \exists k' \forall^* y \in B^{D(M)}_{1/k'}(\bar{a}') \forall m \, (\pi(y) \in A_{n,m}) \Leftrightarrow \]
	\[ \exists n \exists \bar{a}' \in \NN^{<\NN} \exists k' \forall m \forall^* y \in B^{D(M)}_{1/k'}(\bar{a}') \, (\pi(y) \in A_{n,m}) \Leftrightarrow \]
	\[ \exists n \exists \bar{a}' \in \NN^{<\NN} \exists k' \forall m \forall \bar{a}'' \in \NN^{<\NN} \forall k'' \exists^* y \in B^{D(M)}_{1/k''}(\bar{a}'') \big( d(\bar{a},\bar{a}')<1/k \mbox{ and } \]
	\[ (d(\bar{a}',\bar{a}'') \geq 1/k \mbox{ or } \pi(y) \in A_{n,m}) \big)   \Leftrightarrow \]
	\[ \exists n \exists \bar{a}' \in \NN^{<\NN} \exists k' \forall m \forall \bar{a}'' \in \NN^{<\NN} \forall k'' (M \in B^{\Delta \bar{a}',k'}_{\bar{a},\bar{a}',k',n,m}), \]
	where $B_{\bar{a},\bar{a}',k',n,m} \in \bSigma^0_{\beta_{n,m}}(t)$. By the inductive hypothesis, there are fragments $F_{\bar{a}',k'}$ such that $B^{\Delta \bar{a}',k'}_{\bar{a},\bar{a}',k',n,m} \in \bSigma^0_\alpha(t_{F_{\bar{a}',k'}})$. Therefore  $A^{\Delta u,k} \in \bSigma^0_\alpha(t_F)$, where $F$ is the fragment generated by all $F_{\bar{a}',k'}$.
	
\end{proof}

Since $A^*=A$, if $A=[M]$, the above gives the following:

\begin{corollary}
	\label{co:Sigma0alpha}
	Let $L$ be a signature, and let $T$ be a theory such that $\cong_T$ is potentially $\bPi^0_\alpha$. There exists a fragment $F$ such that $[M] \in \bPi^0_\alpha(t_F)$ for every $M \in \Mod(T)$.
\end{corollary}

The following fact is due to Todor Tsankov.

\begin{proposition}
	\label{pr:atomic-to-categorical}
Let $L$ be a signature.	For every fragment $F$, there exists a fragment $F' \supseteq F$ such that $\Th_{F'}(M)$ is $\aleph_0$-categorical for every $F$-atomic model $M \in \Mod(L)$.
\end{proposition}

\begin{proof}
	By the uniqueness of atomic models, it suffices to find an $\IL(L)$ sentence that expresses that the model is $F$-atomic. We claim that the following works:
	\begin{equation}
		\label{eq:atomic}
		\bigvee_n \sup_{z = (z_1, \ldots, z_n)} \bigwedge_{\psi \in F_1} \sup_x \big( (1 \dotdiv\psi(x)) \wedge   \bigvee_{\phi \in F_1} |\phi(x) - \phi(z)|  \big) = 0.
	\end{equation}

	We will check that \eqref{eq:atomic} holds in a structure $M$ iff for every $n$ and every $\bar{c} \in M^n$, $\tp (\bar{c})$ is isolated, i.e., by \cite[Lemma 7.4]{BeDoNiTs}, that for every $\delta > 0$, there exists $\psi \in F$ such that $[\psi < 1] \subseteq B_\delta(\tp (\bar{c}))$ (calculated in $S_n(\Th_F(M))$). Put $T = \Th_F(M)$. The ``if'' direction is clear; we check the converse.
	
	Suppose \eqref{eq:atomic} holds in $M$, and fix $n \in \NN$, $\bar{c} \in M^n$ and $\delta < 1/2$. Let $\psi$ be such that the value of the remaining formula is less than $\delta$. We will show that $[\psi < 1/2] \subseteq B_\delta(\tp (\bar{c}))$, i.e., for all $q \in S_n(T)$
	\begin{equation}
		\label{eq:atomic-categorical-1}
		\psi(q) < 1/2 \mbox{ implies } \partial(q, \tp (\bar{c})) \leq \delta.
	\end{equation}
	As the set of $q \in S_n(T)$ that satisfy \eqref{eq:atomic-categorical-1} is $\tau$-closed and the set of types realized in $M$ is $\tau$-dense, it suffices to check \eqref{eq:atomic-categorical-1} for all $q \in \Theta_n(M)$. Let $\bar{a} \in M^n$, $q = \tp (\bar{a})$ and suppose that $\psi^M(\bar{a}) < 1/2$. Then
	
	\begin{equation*}
		\partial(\tp (\bar{a}), \tp (\bar{c})) = \sup_{\phi \in F_1} |\phi(\bar{a}) - \phi(\bar{c})| \leq \delta,
	\end{equation*}
	which finishes the proof.
\end{proof}
 
\begin{lemma}
	\label{le:ReductiontoPi2}
	Let $F$ be a fragment, and let $T$ be an $F$-theory all of whose models are locally compact. Suppose that $M \in \bPi^0_{\alpha+2}(t_F)$ for some $M \in \Mod(T)$, $\alpha \geq 1$. There is a fragment $F_M \supseteq F$ such that $[M] \in \bPi^0_{2}(t_{F_M})$, and $\rk_F(F_M)=\alpha$. 
\end{lemma}

\begin{proof}
	Assume that $\alpha<\omega$.
	
	Case 1: $\alpha$ is even. Let $P_M$ be the $(\alpha-1)$-AE family as in Corollary \ref{co:IsoAEAlpha}. We will find a fragment $F_0$ with $\rk_F(F_1)=2$, and an $(\alpha-3)$-AE family $P_0$ in $F_0$ such that $N \in \Mod(T)$ models $P_0$ iff $N$ models $P_M$.
	
	Fix a tuple $\bar{a} \in \Seq$ in $M$, and $u \in \QQ^+$. For $\bar{b} \in B_{u}(\bar{a})$, $\bar{b} \in \Seq$, $\epsilon \in \QQ^+$, fix $q_{\bar{b},\epsilon}(\bar{y}) \in B_{\epsilon}(\tp(\bar{b}))$, let $L_{\bar{b},\epsilon}$ be the set of all $1$-Lipschitz formulas $\phi$ in $F$ such that $\phi(q_{\bar{b},\epsilon})=0$, and let
	$$\phi_{\bar{b},\epsilon}=\bigvee L_{\bar{b},\epsilon}.$$
	As $\phi \in L_{\bar{b},\epsilon}$ are $1$-Lipschitz, each $\phi_{\bar{b},\epsilon}$ is a $1$-Lipschitz formula in $\IL$, and a tuple $\bar{b}'$ in $N \in \Mod(T)$ realizes $q_{\bar{b},\epsilon}(\bar{y})$ iff $\phi_{\bar{b},\epsilon}^N(\bar{b}')=0$. Define $1$-Lipschitz formulas
	$$\psi^1_{\bar{a},u}(\bar{x})=\sup_{\bar{y}} [(u \dotdiv d(\bar{x},\bar{y})) \wedge \bigwedge_{\bar{b}, \epsilon} (\phi_{\bar{b},\epsilon}(\bar{y})+\epsilon)],$$
	$$\psi^2_{\bar{a},u}(\bar{x})=\bigvee_{\bar{b},\epsilon} \inf_{\bar{y}} [(d(\bar{x},\bar{y}) \dotdiv u) \vee (\phi_{\bar{b},\epsilon}(\bar{y})+\epsilon)],$$
	$$\psi_{\bar{a},u}=\psi^1_{\bar{a},u} \vee \psi^2_{\bar{a},u}.$$
	
	Fix $N \in \Mod(T)$, and tuple $\bar{a}'$ in $N$. Clearly, if $T^1_{u}(\bar{a})=T^1_{u}(\bar{a}')$, then $\psi^N_{\bar{a},u}(\bar{a}')=0$. On the other hand, if $(\psi^1_{\bar{a},u})^N(\bar{a}')=0$, then for every $\bar{b}' \in B_{u}(\bar{a}')$, and $\epsilon>0$, there is $\bar{b}$ such that $\phi_{\bar{b},\epsilon}^N(\bar{b}')<\epsilon$. And if $(\psi^2_{\bar{a},u})^N(\bar{a}')=0$, then for every $\bar{b}$, and $\epsilon>0$, there is $\bar{b}' \in B_{u}(\bar{a}')$ such that $\phi_{\bar{b},\epsilon}^N(\bar{b}')<\epsilon$.  Moreover, $$\partial(\tp^N(\bar{b}'),q_{\bar{b},\epsilon})<\epsilon,$$ hence $$\partial(\tp^N(\bar{b}'),\tp^M(\bar{b}))<2\epsilon.$$ By local compactness of $N$, it follows that $T^1_{u}(\bar{a})=T^1_{u}(\bar{a}')$. Thus, $$T^1_{u}(\bar{a})=T^1_{u}(\bar{a}') \mbox{ iff } \psi^N_{\bar{a},u}(\bar{a}')=0$$ for every tuple $\bar{a}'$ in $N \in \Mod(T)$. 
	
	Let $F_0$ be the fragment generated by $F$ and $\psi_{\bar{a},u}(\bar{x})$, $\bar{a} \in \NN^{<\NN}$, $u \in \QQ^+$. Fix a bijection $\langle\cdot,\cdot\rangle:\NN \times \QQ^+ \rightarrow \NN$. We construct an $(\alpha-3)$-family $P_0(\bar{x})$, by replacing every $3$-AE family $Q(\bar{x})=\{q_{k,l}(\bar{x}_{k,l})\}$ appearing in $P_M(\bar{x})$ with a $1$-AE family $Q'(\bar{x})=\{q'_{\langle k,v\rangle,l}(\bar{x}_{\langle k,v\rangle,l})\}$, where, for any fixed $k$, and $v \in \QQ^+$, $\{q'_{\langle k,v\rangle,l}\}$ enumerates all $\psi_{\bar{c},v}$ that come from $\bar{c}$ in $M$ and $v>0$ witnessing realizations of $Q(\bar{x})$. Obviously, $M$ models $P_0$. And if $N \in \Mod(T)$ models $P_M$, it is isomorphic with $M$, hence models $P_0$. On the other hand, if $\bar{c}'$ realizes some $q'_{\langle k,v\rangle,l}(\bar{x}_{\langle k,v\rangle,l})$, there is $\bar{c}$ in $M$ witnessing a realization of $Q(\bar{x})$, and such that $T^1_{v}(\bar{c})=T^1_{v}(\bar{c}')$. By Proposition \ref{pr:AET}, $\bar{c}'$ realizes some $q_{k,l'}(\bar{x})$. And $[M] \in \bPi^0_{\alpha-2}(t_{F_0})$, since $[M]$ is characterized by an $(\alpha-3)$-family. Finally, in order to get the required $F_M$, we iterate the above construction sufficiently many times.
	
	Case 2: $\alpha$ is odd. Consider $F_1$ generated by $F$ and formulas $\phi_{\bar{c},\epsilon}$ as above for $\bar{c} \in \Seq$, $\epsilon \in \QQ^+$. Clearly, $\rk_F(F_1)=1$. We construct an $(\alpha-2)$-AE family $P_1(\bar{x}$) in $F_1$ by replacing every $2$-AE family $Q(\bar{x})=\{q_{k,l}(\bar{x}_{k,l})\}$ appearing in $P(\bar{x})$ with $Q'(\bar{x})=\{q'_{\langle k,v\rangle,l}(\bar{x}_{\langle k,v\rangle,l})\}$, where for any fixed $k$, and $v \in \QQ^+$, $\{q'_{\langle k,v\rangle,l}\}$ enumerates all $\phi_{\bar{c},v}$ that come from $\bar{c}$ in $M$ and $v>0$ witnessing realizations of $Q(\bar{x})$. As before, $N \in \Mod(T)$ models $P_1$ iff $N$ models $P_M$, and $[M] \in \bPi^0_{\alpha-1}(t_{F_0})$. In this way, Case 2 can be reduced to Case 1.
	
	Finally, an easy induction using the above arguments reduces the case  $\alpha \geq \omega$ to the case $\alpha<\omega$.
\end{proof}

\begin{theorem}
	\label{th:Potentially}
	Let $F$ be a fragment, and let $T$ be an $F$-theory all of whose models are locally compact. Suppose that $\cong_{T}$ is potentially $\bPi^0_{\alpha+2}$, where $\alpha \geq 1$. Then $\cong_{T}$ is Borel reducible to $=_{\alpha+1}$.
\end{theorem}

\begin{proof}
Observe that for $M \in \Mod(T)$, the fragment $F_M$ given by Lemma \ref{le:ReductiontoPi2} can be coded as an element of $\mathcal{P}^{\alpha}(\NN)$. First, it is not hard to see that, with an aid of the Kuratowski--Ryll-Nardzewski theorem, selecting  the types $q_{\bar{b},u}$ in the proof of the lemma can be arranged in a Borel and isomorphism invariant way. Moreover, the fragments $F_0$ and $F_1$ are also constructively specified, given $q_{\bar{b},u}$'s, so it is somewhat tedious but completely standard to verify that the assignment $M \mapsto F_M$ can be done in a Borel and isomorphism invariant manner. 
	
Now, by \cite[Theorem 4.3]{HaMaTs}, each $M$ is an $F_M$-atomic model of $\Th_{F_M}(M)$, so it is $\aleph_0$-categorical in the theory $\Th_{F'_M}(M)$, where $F'_M$ is the fragment given by Proposition \ref{pr:atomic-to-categorical}. As $\Th_{F'_M}(M)$ can be regarded as an element of $\mathcal{P}^{\alpha+1}(\NN)$, the assignment $M \mapsto (F'_M,\Th_{F'_M}(M))$ can be coded as a Borel mapping $f:\Mod(T) \rightarrow \mathcal{P}^{\alpha+1}(\NN)$ reducing $\cong_{T}$ to $=_{\alpha+1}$.
\end{proof}

%\begin{proof}
%The relation $R \subseteq \Mod(L) \times \NN^{<\NN} \times \mathcal{U} \times \QQ^+$, defined by $(M,\bar{a},U,\epsilon) \in R$ iff $(U,\epsilon)$ is $\bar{a}$-good in $M$, is Borel by \cite{???}. Also, for any given $\bar{a}$, $U$ and $\epsilon>0$, the set  $\cl[\tau]{B_\epsilon(\tp(\bar{a})) \cap U}$ can be coded as a binary sequence (recording basic open subsets in the complement), and it is easy to verify that this assignement can be done in a Borel way. Thus, for any $M$, the sequence of such codes (as $\bar{a}$, $U$ and $\epsilon>0$ vary) can be also assigned in a Borel manner. An easy induction on $\alpha$ shows that for every $\alpha>0$ the assignement $M \mapsto x_M$, where $x_M$ is a code for $T^\alpha(M)$ can be Borel. And, for every fixed $\alpha<0$, there is a natural action of a subgroup of $G \leq \SI$ such that $g.x_M=x_N$ iff $T^\alpha(M)=T^\alpha(N)$.
%\end{proof}

\section{Countable structures and approximations of the Hjorth-isomorphism game}

\paragraph{\textbf{Countable structures}}
Classical countable structures can be recovered in the setting of Polish metric structures by imposing the requirement that $d(x,y)=1$ for $x \neq y$, which can be axiomatized by the condition
\[ \sup_{x,y}(d(x,y) \wedge (d(x,y) \dotdiv 1) \wedge (1 \dotdiv d(x,y)))=0; \]
in the same way one can make sure that predicates take only values $0$ or $1$. For such metric $d$, quantifiers $\sup_x$ and $\inf_y$ behave as $\forall x$ and $\exists x$, so $\FL$, $\IL$, and topologies defined by fragments are as in the classical setting. In particular, $\Mod(L)$ is the space of all structures in signature $L$ and with universe $\NN$, and, for $u<1$, $d(\bar{a},\bar{b})<u$ iff $\bar{a} \subseteq \bar{b}$ or $\bar{b} \subseteq \bar{a}$, so it suffices to consider only AE families $P(\bar{x})$ with $u_P=1$. Observe that then a tuple $\bar{a}$ in $M$ realizes a $(\beta+n)$-AE family $P(\bar{x})=\{p_{k,l}(\bar{x}_{k,l})\}$, $1 \leq n< \omega$, iff
\[  \forall  \bar{b} \supseteq \bar{a} \forall k  \exists \bar{c} \supseteq \bar{b} \exists l  (\bar{c} \mbox{ realizes } p_{k,l}(\bar{x}_{k,l}) \mbox{ in } M ).  \] 
As a matter of fact, Theorem \ref{th:isoTalpha} also takes a more transparent form. Fix a fragment $F$ in signature $L$. For $M \in \Mod(L)$, and $\bar{a} \in \Seq$, we define $\tp^0(\bar{a})=\tp(\bar{a})$, and, for $\alpha>0$,
\[ \tp^\alpha(\bar{a})=\{ \tp^\beta(\bar{b}) : \beta<\alpha, \, \bar{b} \in \Seq, \, \bar{a} \subseteq \bar{b}  \}.  \]
We also put $\Th^\alpha(M)=\tp^\alpha(\emptyset)$. If $M$ or $F$ is not clear from the context, we may explicitly specify it by writing $\tp^\alpha_M(\bar{a})$, $\tp^\alpha_F(\bar{a})$ or $\Th^\alpha_F(M)$. Clearly, $\Th^0(M)=\Th(M)$, $\Th^1(M)$ is the collection of all $F$-types realized in $M$, $\Th^2(M)$ is the collection of all $F$-types of structures $(M,\bar{a})$, $\bar{a} \in \Seq$, etc.

\begin{theorem}
	\label{th:isoTheoryalpha}
	Let $F$ be a fragment in signature $L$. Suppose that $[M] \in \bPi^0_{1+\alpha}(t_F)$, $\alpha \geq 1$, for some $M \in \Mod(L)$. Then $$[M]=\{ N \in \Mod(L): \Th^{\alpha}_F(N)=\Th^{\alpha}_F(M)\}.$$
\end{theorem}

\begin{proof}
We show that $\Th^\alpha_F(M)=\Th^\alpha_F(N)$ implies that $M$ and $N$ model the same $\alpha$-AE families $P(\emptyset)$, and apply Corollary \ref{co:IsoAEAlpha}. For $\alpha=1$, suppose that $M$ and $N$ realize the same types, and let $P(\emptyset)=\{p_{k,l}(\bar{x}_{k,l})\}$ be a $1$-AE family. Let $\bar{b}$, $\bar{b}'$ be tuples in $M$, $N$, respectively, such that $\tp(\bar{b})=\tp(\bar{b}')$. Then, for every $k$ and $l$, there is $\bar{c}$ such that the formula $p_{k,l}(\bar{b},\bar{c})$ holds in $M$ iff there is $\bar{c}'$ such that $p_{k,l}(\bar{b}',\bar{c}')$ holds in $N$. For $\alpha=2$ the argument is analogous, and for $\alpha>2$ this is an easy induction.
\end{proof}

\paragraph{\textbf{Approximations of the Hjorth-isomorphism game}}
A \emph{Polish $G$-space} $X$ is a continuous action of a Polish group $G$ on a Polish space $X$, $E_X$ denotes the orbit equivalence relation induced by $X$, and $[x]$ is the equivalence class of $x \in X$. In \cite{LuPa}, a game-theoretic approach to Hjorth's theory of turbulence has been developed, giving rise to an interesting sufficient condition for orbit equivalence relations not to be classifiable by countable structures. In this section, we introduce a hierarchy of games $\Appr_\alpha(x,y)$, $\alpha<\omega_1$, that are finer and finer approximations of the Hjorth-isomorphism game $\Iso(x,y)$ from \cite{LuPa}. As it turns out, winning strategies in these games, played for the logic $\Sym(\NN)$-spaces $\Mod(L)$, are related to families $\Th^\alpha(M)$. Moreover, when put together, they also can be used to rule out classifiability by countable structures. 

For a Polish $G$-space $X$, $x,y \in X$, a collection $\VV$ of open neighborhoods of $1$ in $G$, an open neighborhood $V$ of $1$ in $G$, and open $U \subseteq X$, we define games $\Appr_\alpha(x,y,\VV,V,U)$, $\alpha<\omega_1$. Fix $\alpha<\omega_1$, set $x_0=x$, $y_0=y$, $V^y_0=V$, $U^y_0=U$, and let Odd and Eve play as follows.

(1) In the first turn, Odd either sets $V^x_1=V^y_0$, $\alpha_1=\alpha$ or plays a new $V^x_1 \in \VV$, and $\alpha_1<\alpha$. Then he plays an open neighborhood $U^x_0$ of $x_0$. Eve replies with $g^y_0 \in G$.

(2) In the second turn, Odd either sets $V^y_1=V^x_0$, $\alpha_{2}=\alpha_1$ or plays a new $V^y_1 \in \VV$, and $\alpha_2<\alpha_1$. Then he plays an open neighborhood $U^y_1$ of $y_1$. Eve replies with $g^x_0 \in G$.

(2n+1) In the $(2n+1)$-th turn, $n>0$, Odd either sets $V^x_n=V^y_{n}$, $\alpha_{2n+1}=\alpha_{2n}$ or plays a new $V^x_n \in \VV$, and $\alpha_{2n+1}<\alpha_{2n}$. Then he plays an open neighborhood $U^x_{n}$ of $x_{n}= g^x_{n-1}.x_{n-1}$. Eve replies with $g^y_{n} \in G$.

(2n) In the $(2n)$-th turn, $n>1$, Odd either sets $V^y_n=V^x_{n-1}$, $\alpha_{2n}=\alpha_{2n-1}$ or plays a new  $V^y_n \in \VV$, and $\alpha_{2n}<\alpha_{2n-1}$. Then he plays an open neighborhood $U^y_{n}$ of $y_{n}= g^y_{n-1}.y_{n-1}$. Eve replies with $g^x_{n-1} \in G$.

The game proceeds in this way, producing elements $V^x_n$, $V^y_n$, $x_n$, $y_n$, $g^x_n$, $g^y_n$, $U^x_n$ and $U^y_n$, $n \in \NN$. Eve wins if, for every $n \geq 0$,

\begin{itemize}
	\item $y_{n+1} \in \overline{U^x_n}$ and $x_n \in \overline{U^y_n}$,
	\item $g^y_n= h_k \ldots h_0$ for some $k \geq 0$ and $h_0, \ldots, h_k \in V^y_n$ such that $h_i \ldots h_0.y_n \in U^y_n$ for $i \leq k$,
	\item $g^x_n= h_k \ldots h_0$ for some $k \geq 0$ and $h_0, \ldots, h_k \in V^x_n$ such that $h_i \ldots h_0.x_n \in U^x_n$ for $i \leq k$.
\end{itemize}

We write shortly $\Appr_\alpha (x,y)$ for $\Appr_\alpha (x,y,\UUU,G, X)$, where $\UUU$ is the collection of all open neighborhoods of $1$ in $G$. We write $x \sim_{\alpha} y$ if Eve has a winning strategy in $\Appr_\alpha(x,y)$.

Note that if the conditions regulating the choice of $\alpha$ in $\Appr_\alpha(x,y)$ are dropped, i.e., Odd can play a new $V^x_n$ (or $V^y_y$) without having to select some $\alpha_{2n+1}<\alpha_{2n}$ (or $\alpha_{2n}<\alpha_{2n-1}$), the resulting game is the Hjorth-isomorphism game $\Iso(x,y)$ defined in \cite{LuPa}. In other words, the games $\Appr_\alpha(x,y)$, $\alpha<\omega_1$, form a hierarchy of finer and finer approximations of $\Iso(x,y)$.

As the group $\Sym(\NN)$ of all permutations of natural numbers has a neighborhood basis at $1$ consisting of subgroups, we have the following:

\begin{remark}
	Suppose that $G \leq\Sym(\NN)$.
	\label{re:GameForSI}
	\begin{enumerate}
		\item In terms of winning strategies, the requirements for Eve in $\Appr_\alpha(x,y)$ reduce to
		\begin{itemize}
			\item $y_{n+1} \in \overline{U^x_n}$ and $x_n \in \overline{U^y_n}$,
			\item $g^y_{n} \in V^y_n$,
			\item $g^x_{n} \in V^x_n$.
		\end{itemize}
	\item If $V \leq G$, $\Appr_0(x,y,\{V\},G,X)$ is $\Appr_{G,V}(x,y)$ defined in \cite{KeMaPaZi}.
	\end{enumerate}
	
\end{remark}

\begin{proposition}
	Let $X$ be a $G$-space, and let $\alpha<\omega_1$. Then $x \in [y]$ implies $x \sim_{\alpha} y$ for $x,y \in X$, and $\sim_{\alpha}$ is an equivalence relation.
\end{proposition} 

\begin{proof}
	
	The first statement, and symmetry of $\sim_{\alpha}$, are obvious. %Suppose that $x \sim_{\HH,1} y$, and let us start the game $\Appr_1(y,x)$. In the first turn, Odd plays $U^y_0 \subseteq X$ and $V^y_1 \subseteq G$. Applying her winning strategy for $\Appr_1(x,y)$, Eve can find $h^y_0, h^x_0 \in G$ so that  $h^x_0.x \in W$ for a neighborhood $W$ of $h^y_0.y$ with $(h^y_0)^{-1}.W \subseteq U^y_0$. Then $(h^y_0)^{-1}h^x_0.x \in U^y_0$, and Eve plays $g^x_0=(h^y_0)^{-1}h^x_0$ in $\Appr_1(y,x)$. In the second turn, Odd plays $U^x_1 \subseteq X$ and $V^x_1 \subseteq G$. As before, for a fixed neighborhood $W$ of $g^x_0.x$ such that $(h^y_0)^{-1}.W \subseteq U^x_1$, Eve can find $h^y_1$ such that $(h^y_0)^{-1}h^y_1h^y_0.y \in U^x_1$. Hence, as long as $(h^y_0)^{-1}h^y_1h^y_0 \in V^y_1$, which can also be warranted because Eve has a winning strategy for $\Appr_1(x,y)$, she has not lost yet. In this fashion, we can construct a winning strategy for Eve for the entire game $\Appr_1(x,y)$, i.e., $y \sim_{\HH, } x$.
	
	We prove transitivity for $\alpha=1$. Suppose that $x \sim_{1} y$ and $y \sim_{1} z$. We show that $z \sim_{\alpha} x$, which, by symmetry of $\sim_{\alpha}$, implies that $x \sim_{\alpha} z$.  In the first turn, Odd plays $U^z_0 \subseteq X$ and $V^z_1 \subseteq G$. Applying her winning strategy for $\Appr_1(z,y)$, Eve can find $h^y_0$ such that $h^y_0.y \in U^z_0$. Then, applying her winning strategy for $\Appr_1(y,x)$, for a neighborhood $W$ of $y$ with $h^y_0.W \subseteq U^z_0$, she can find $h^x_0$ such that $h^x_0.x \in W$, i.e., $h^y_0h^x_0.x \subseteq U^z_0$. She plays $g^x_0=h^y_0h^x_0$. In the second turn, Odd plays $U^x_1 \subseteq X$ and $V^x_1 \subseteq G$. Eve first applies her winning strategy for $\Appr_1(y,x)$ to find $(h^y_0)'$ such that $(h^y_0)'.y \in W$ for a neighborhood $W$ of $h^x_0.x$ such that $h^y_0.W \subseteq U^x_1$. Then she applies her winning strategy for $\Appr_1(y,z)$ to find $h^z_0$ such that $h^z_0.z \in W'$ for a neighborhood $W'$ of $h^y_0.y$ such that $h^y_0(h^y_0)'(h^y_0)^{-1}.W' \subseteq W$. Clearly, $h^y_0(h^y_0)'(h^y_0)^{-1}.z \in U^x_1$, so Eve plays $g^z_0=h^y_0(h^y_0)'(h^y_0)^{-1}$ in $\Appr_1(z,x)$. Proceeding in this way, we can construct a winning strategy for Eve for the entire game $\Appr_1(z,x)$, i.e., $z \sim_{1 } x$.
	
	The inductive step is straightforward: every game $\Appr_\alpha(x,y)$ proceeds initially as $\Appr_1(x,y)$, and then as some $\Appr_\beta(x,y)$, where $\beta<\alpha$.
\end{proof}

For a given signature $L$, the logic $\Sym(\NN)$-space $\Mod(L)$ is the logic action of $\Sym(\NN)$ on $\Mod(L)$ permuting the universe of structures $M \in \Mod(L)$. Clearly, $E_{\Mod(L)}$ is the isomorphism relation on $\Mod(L)$. As the relations $\sim_\alpha$ depend on the topology on $\Mod(L)$, for a fragment $F$, we will write $\sim_{\alpha,F}$ to denote $\sim_\alpha$ on $(\Mod(T),t_F)$, and $[M]_{\alpha,F}$ to denote equivalence classes of $\sim_{\alpha, F}$.

\begin{proposition}
	Let $F$ be a fragment in signature $L$, and let $M \in \Mod(L)$. Then $$[M]_{\alpha,F}=\{N \in \Mod(L): \Th^{\alpha}_F(N)=\Th^{\alpha}_F(M)\}.$$
\end{proposition}

\begin{proof}
	For $m \in \NN$, let $\bar{m}$ denote the tuple $(0, \ldots m-1)$. For $\alpha=1$, suppose that $\Th^1(M)=\Th^1(N)$ for some $M,N \in \Mod(L)$, i.e., $M$ and $N$ realize the same types. Then Eve has a winning strategy along the following lines. Without loss of generality, we can assume that in the first turn Odd chooses  $[\phi(\bar{m},\bar{a})]$ as $U^M_0$, where $\bar{m}$ and $\bar{a}$ are disjoint, and the pointwise stabilizer $V_m$ of $\bar{m}$ as $V^M_1$. Suppose that $\tp_M(\bar{m})=\tp_N(\bar{m})$. Eve fixes $\bar{c} \in \Seq$ witnessing that $N \models \exists \bar{x} \phi(\bar{n},\bar{x})$, and chooses $g_0^N \in V_m$ mapping $\bar{c}$ to $\bar{a}$. Clearly, $g_0.N \in [\phi(\bar{m},\bar{a})]$. Otherwise, since $\Th^1(M)=\Th^1(N)$, there is $\bar{b}$ such that $\tp_M(\bar{m})=\tp_N(\bar{b})$, so, arguing as before, Eve can find $g_0^N \in \Sym(\NN)$ such that $g_0.N \in [\phi(\bar{m},\bar{a})]$. Other turns are analogous. In particular, for $n>1$, the elements $g_n^M$ or $g_n^N$ can be always chosen from $V_m$.
	
	On the other hand, suppose that $\Th^1(M) \neq \Th^1(N)$, say, there is $m$ such that no tuple in $N$ realizes $\tp_M(\bar{m})$. Let Odd choose $V_m$ in the first turn, and let $g_0^N$ be the element chosen by Eve. Then, for $\bar{b}=(g_0^N)^{-1}(\bar{m})$, there exists $\phi(\bar{x}) \in \tp_M(\bar{m}) \bigtriangleup \tp_N(\bar{b})$. It is not hard to see that without loss of generality, we can assume that  $\phi(\bar{x})$ is of the form $\exists \bar{y} \psi(\bar{x},\bar{y})$ or $\neg \exists \bar{y} \psi(\bar{x},\bar{y})$. Thus, there exists $\bar{a}$ such that $\psi(\bar{m},\bar{a})$ holds in one of the structures, while for no $\bar{c}$, $\psi(\bar{b},\bar{c})$ holds in the other one. In other words, either $M \in [\psi(\bar{m},\bar{a})]$, while there is no $g \in V_m$ such that $gg_0^N.N \in [\psi(\bar{m},\bar{a})]$, or $N \in [\psi(\bar{m},\bar{a})]$, while there is no $g \in V_m$ such that $g.M \in [\psi(\bar{m},\bar{a})]$. In any case, by Remark \ref{re:GameForSI}, Eve looses the game.
	
	For the inductive step, we assume first that, for every $\beta<\alpha$ and $m$, $\tp^\beta_M(\bar{m})=\tp^\beta_N(\bar{m})$ iff Eve has a winning strategy starting with some $g^N_0 \in V_m$. Then we proceed as above.   
\end{proof}

\begin{corollary}
	\label{co:Iso=GameIso}
	Let $F$ be a fragment in signature $L$. Suppose that $[M] \in \bPi^0_{1+\alpha}(t_F)$,  $\alpha \geq 1$, for some $M \in \Mod(L)$. Then $[M]=[M]_{\alpha,F}$.
\end{corollary}

For equivalence relations $E$, $F$ on Polish spaces $X$, $Y$, respectively, an \emph{$(E,F)$-homomorphism} is a mapping $f:X \rightarrow Y$ such that  $$x_1 E x_2 \mbox{ implies } f(x_1) F f(x_2).$$ Analogously to the Hjorth-isomorphism relation, one can show that Baire-measurable homomorphisms preserve the relations $\sim_\alpha$ on a comeager set.

\begin{proposition}
	\label{pr:GraphHomo}
	Let $X$ be a Polish $G$-space, $Y$ a Polish $H$-space, and let $f$ be a Baire-measurable $(E_X,E_Y)$-homomorphism. For every $\alpha<\omega_1$ there exists a $G$-invariant comeager subset $X_0 \subseteq X$ such that $x  \sim_{\alpha} y$ implies $f(x)  \sim_{\alpha} f(y)$ for $x,y \in X_0$.
\end{proposition}

\begin{proof}
	As it has been pointed out in Remark \ref{re:GameForSI},  each $\Appr_\alpha(x,y)$ is the game $\Iso(x,y)$ with the extra ingredient of selecting (smaller and smaller) ordinals whenever a new neighborhood of the identity is played by Odd. Therefore the proof of the proposition is essentially the same as the proof of Proposition 3.6 in \cite{LuPa}, which states the same fact for $\Iso(x,y)$. One only needs to make the following straightforward observation when constructing a winning strategy for Eve in $\Appr_\alpha(f(x),f(y))$ based on her winning strategy in $\Appr_\alpha(x,y)$: as long as no new neighborhood of $1$ has been played by Odd in $\Appr_\alpha(f(x),f(y))$, no new neighborhood of $1$ is played by Odd in $\Appr_\alpha(x,y)$.
\end{proof}

\begin{theorem}
	\label{th:NotClassifiablePialpha}
	Let $X$ be a Polish $G$-space, and let $\alpha<\omega_1$. If for any $G$-invariant comeager subset $C$ of $X$ there exist $x, y \in C$ such that $x \sim_{\alpha} y$ but $[x] \neq [y]$, then there is no Borel reduction of a restriction of $E_X$ to a comeager $X_0 \subseteq X$ to a potentially $\bPi^0_{1+\alpha}$ isomorphism relation of the form $\cong_T$ for some fragment $F$ and $F$-theory $T$.
\end{theorem}

\begin{proof}
%As $X_0^*$ is comeager, for any comeager $X_0 \subseteq X$, we can assume that $X_0$ in the statement of the theorem is invariant.
Suppose that there is an $F$-theory $T$, comeager $X_0 \subseteq X$, and a Borel reduction $f$ of the restriction of $E_{X}$ to $X_0$, to $\cong_T$  such that $\cong_T$ is potentially $\bPi^0_{1+\alpha}$. By Corollary \ref{co:Sigma0alpha}, we can assume that $M \in \bPi^0_{1+\alpha}(t_F)$ for $M \in \Mod(T)$. By Proposition \ref{pr:GraphHomo}, there is a comeager $C \subseteq X_0$ such that $x \sim_\alpha y$ implies $f(x) \sim_{\alpha,F} f(y)$ for $x,y \in C$. Fix $x,y \in C$ such that $x \sim_{\alpha} y$ but $[x] \neq [y]$. But then $f(x) \sim_{\alpha, F} f(y)$, and, by Corollary \ref{co:Iso=GameIso}, $f(x) \cong f(y)$, a contradiction. 
\end{proof}

\begin{theorem}
	\label{th:NotClassifiable}
	Let $X$ be a Polish $G$-space. If for any $\alpha<\omega_1$, and any $G$-invariant comeager subset $C$ of $X$ there exist
	$x, y \in C$ such that $x \sim_{\alpha} y$ but $[x] \neq [y]$, then no restriction of $E_X$ to a comeager $X_0 \subseteq X$ is classifiable by countable structures.
\end{theorem}

\begin{proof}
Suppose that there is a comeager $X_0 \subseteq X$, and a Borel reduction $f$ of a restriction of $E_X$ to $X_0$, to $\cong$ on $\Mod(L)$ for some signature $L$. By Claim 5.4 in the proof of \cite[Theorem 1.3]{AlPa}, there is a comeager $C'\subseteq X_0$ and $\alpha<\omega_1$ such that for every $x \in C'$, $[f(x)]$ is $\bPi^0_\alpha$ in the standard topology on $\Mod(L)$, and so also in the topology generated by the finitary fragment $\FL$. By Proposition \ref{pr:GraphHomo}, there is a comeager $C \subseteq C'$ such that $x \sim_\alpha y$ implies $f(x) \sim_{\alpha,\FL} f(y)$ for $x,y \in C$. Fix $x,y \in C$ such that $x \sim_{\alpha} y$ but $[x] \neq [y]$. But then $f(x) \sim_{\alpha,\FL} f(y)$, and, by Corollary \ref{co:Iso=GameIso}, $f(x) \cong f(y)$, a contradiction.
\end{proof}


\begin{thebibliography}{99}
\bibitem{AlPa} S. Allison, A. Panagiotopoulos, \emph{Dynamical obstructions to classification by (co)homology and other TSI-group invariants}, Trans. Amer. Math. Soc. 374 (2021), 8793--8811.
\bibitem{BeDoNiTs} I. Ben Yaacov, M. Doucha, A. Nies, T. Tsankov, \emph{Metric Scott analysis}, Adv. Math. 318 (2017), 46--87.
\bibitem{BeIo} I. Ben Yaacov, J. Iovino, \emph{Model theoretic forcing in analysis}, Ann. Pure Appl. Logic 158 (2009), 163--174.
%\bibitem{BeBeHeUs}  I. Ben Yaacov, A. Berenstein, C.W. Henson, A. Usvyatsov, \emph{Model theory for metric structures}, London Mathematical Society Series Notes, vol. 350, Cambridge University Press 2008, 315--427.
\bibitem{GaKe} S. Gao, A. Kechris, \emph{On the classification of Polish metric spaces up to isometry}. Memoirs of the American Mathematical Society, 161 (2003).
\bibitem{Gao} S. Gao, \emph{Invariant Descriptive Set Theory}, Chapman and Hall 2008.
\bibitem{HaMaTs} A. Hallbäck, M. Malicki, T. Tsankov, \emph{Continuous logic and Borel equivalence relations}, J. Symbolic Logic, published online 22.06.2022, DOI:10.1017/jsl.2022.48
\bibitem{Hj} G. Hjorth, \emph{Classification and orbit equivalence relations}, Mathematical Surveys and Monographs, vol. 75, American Mathematical Society.
\bibitem{HjKe} G. Hjorth, A. Kechris, \emph{Borel equivalence relations and classifications of countable models}, Annals Pure Applied Logic 82 (1996), 221--272.
\bibitem{HjKeLo} G. Hjorth, A. Kechris, A. Louveau, \emph{Borel equivalence relations induced by actions of the symmetric group}, Annals Pure Applied Logic 92 (1998) 63--112.
\bibitem{KeMaPaZi} A. Kechris, M. Malicki, A. Panagiotopoulos, J. Zielinski, On Polish groups admitting non-essentially countable actions. Ergod. Theory Dyn. Syst., 42 (2022), 180--194.
\bibitem{LuPa} M. Lupini, A. Panagiotopoulos, \emph{Games orbits play and obstructions to Borel reducibility}, Groups Geom. Dyn. 12 (2018), 1461--1483.
	\end{thebibliography}
\end{document}